\documentclass{amsart}
\usepackage[utf8]{inputenc}
\usepackage{amsmath}
\usepackage{amsthm}
\usepackage{amsfonts}
\usepackage{amssymb}
\usepackage[all]{xy}
\usepackage{hyperref}
\usepackage{indentfirst}

\newtheorem{thm}{Theorem}[section]

\newtheorem{theorem}[thm]{Theorem}
\newtheorem{corollary}[thm]{Corollary}

\newtheorem{question}[thm]{Question}
\newtheorem{lemma}[thm]{Lemma}

\theoremstyle{definition}

\newtheorem{example}[thm]{Example}
\newtheorem{definition}[thm]{Definition}

\newtheorem{remark}[thm]{Remark}

\DeclareMathOperator{\Ima}{Im}

\DeclareMathOperator{\Span}{span}

\title[Koopman representation of a topological full group]{On the $C^*$-algebra generated by the Koopman representation of a topological full group}
\date{}
\author{Eduardo Scarparo}
\thanks{This work was supported by CNPq, National Council for Scientific and Technological Development - Brazil.}
\address{ Departamento de Matemática, Universidade Federal de Santa Catarina,\newline
Campus Universitário Trindade, 88040-900, Florianópolis - SC, Brazil}
 \email{duduscarparo@gmail.com}
\begin{document}

\begin{abstract}
Let $(X,T,\mu)$ be a Cantor minimal sytem and $[[T]]$ the associated topological full group. We analyze $C^*_\pi([[T]])$, where $\pi$ is the Koopman representation attached to the action of $[[T]]$ on $(X,\mu)$. 

Specifically, we show that $C^*_\pi([[T]])=C^*_\pi([[T]]')$ and that the kernel of the character $\tau$ on $C^*_\pi([[T]])$ coming from containment of the trivial representation is a hereditary $C^*$-subalgebra of $C(X)\rtimes\mathbb{Z}$. Consequently, $\ker\tau$ is stably isomorphic to $C(X)\rtimes\mathbb{Z}$, and $C^*_\pi([[T]]')$ is not AF.

We also prove that if $G$ is a finitely generated, elementary amenable group and $C^ *(G)$ has real rank zero, then $G$ is finite.
\end{abstract}
\maketitle
\section{Introduction}
In this work, we study the real rank zero and AF properties for certain classes of group $C^*$-algebras. The motivations are the classical equivalence between amenability of a group and nuclearity of its $C^*$-algebra, and the equivalence between local finiteness of a group and finiteness of its uniform Roe algebra worked out in \cite{MR2800923}, \cite{MR3158244} and \cite{scarparo_2017}.

For a compact metric space $X$, both $C(X)$ being AF and having real rank zero are equivalent to total disconnectedness of $X$. 

If a group $G$ is countable and locally finite, then $C^*(G)$ is clearly AF (hence it has real rank zero). It is an open problem whether there exists a non-locally finite group $G$ such that $C^*(G)$ is AF. 

In \cite[Theorem 2]{MR1164146}, Kaniuth proved that if $G$ is a nilpotent group and $C^*(G)$ has real rank zero, then $G$ is locally finite. 

In section \ref{rank}, we show that if $G$ is a finitely generated, elementary amenable group, and $C^*(G)$ has real rank zero, then $G$ is finite (Theorem \ref{tempo}). Our proof relies on the fact that infinite, finitely generated, elementary amenable groups virtually map onto $\mathbb{Z}$ \cite[Chapter I, Lemma 1]{MR1275829}.

Let $(X,T,\mu)$ be a Cantor minimal system and $\pi$ the Koopman representation associated to the action of the topological full group $[[T]]$ on $(X,\mu)$.

 Notice that $C^*([[T]])$ does not have real rank zero, since  $[[T]]$ maps onto $\mathbb{Z}$ (by \cite[Theorem 1.1(i)]{MR978619}, or \cite[Proposition 5.5]{MR1710743}). On the other hand, by results of Matui, the commutator subgroup $[[T]]'$ is simple (\cite{MR2205435}) and non-locally finite (this follows from much sharper results from \cite{MR3103094}). Hence, commutators of topological full groups form a class which is not covered by Theorem \ref{tempo}.

Futhermore, it was proven by Juschenko and Monod (\cite{MR3071509}) that $[[T]]$ is amenable. In Section \ref{main}, we prove that $C^*([[T]]')$ is not AF. This is done by showing that $C^*_\pi([[T]])=C^*_\pi([[T]]')$, and that the kernel of the character $\tau$ on $C^*_\pi([[T]])$ coming from containment of the trivial representation is a hereditary $C^*$-subalgebra of $C(X)\rtimes\mathbb{Z}$. Consequently, $\ker\tau$ is stably isomorphic to $C(X)\rtimes\mathbb{Z}$, and $C^*_\pi([[T]]')$ is not AF and has real rank zero.

In Section \ref{odom}, we discuss examples coming from odometers.

\section{Elementary amenable groups and real rank zero}\label{rank}

Recall that the class of elementary amenable groups is the smallest class of groups containing all abelian and all finite groups, and closed under taking subgroups, quotients, extensions and inductive limits. 

We will use the following fact about elementary amenable groups, due to Hillman (\cite[Chapter I, Lemma 1]{MR1275829}). See also \cite[Lemma 1]{MR3586092} for a slightly different proof.

\begin{lemma}[\cite{MR1275829}]\label{aff}

If $G$ is an infinite, finitely generated, elementary amenable group, then there is a subgroup of finite index of $G$ which admits a homomorphism onto $\mathbb{Z}$.
 
\end{lemma}

A $C^*$-algebra $A$ is said to have real rank zero if every hereditary $C^*$-subalgebra of $A$ has an approximate unit of projections (not necessarily increasing). We refer the reader to, for example, \cite[Section V.7]{MR1402012} for other equivalent definitions of real rank zero.

\begin{lemma}\label{blim}
If $A$ is an infinite-dimensional, real rank zero $C^*$-algebra, then it contains a sequence of non-zero, orthogonal projections.
\end{lemma}
\begin{proof}
Since $A$ is infinite-dimensional, there is a sequence $(a_n)_{n\in\mathbb{N}}\subset A$ of non-zero, positive elements such that $a_ja_k=0$ when $j\neq k$ (see, for example, \cite[Exercise 4.6.13]{MR1468229} or \cite{MR0066569}).

 For each $n\in\mathbb{N}$, take a non-zero projection $p_n$  in the hereditary (hence real rank zero) $C^*$-subalgebra $\overline{a_nAa_n}$. By construction, $p_jp_k=0$ when $j\neq k$.
\end{proof}
\begin{theorem}\label{tempo}
If $G$ is a finitely generated, elementary amenable group and $C^*(G)$ has real rank zero, then $G$ is finite.
\end{theorem}
\begin{proof}
Suppose $G$ is infinite. By Lemma \ref{aff}, there is a subgroup $H$ of $G$ with finite index $n$, and $\Phi\colon H\to \mathbb{Z}$ a surjective homomorphism. Let $\varphi\colon C^*(H)\to C^*(\mathbb{Z})$ be the $*$-homomorphism induced by $\Phi$, and $\varphi_n\colon M_n(C^*(H))\to M_n(C^*(\mathbb{Z}))$ the inflation of $\varphi$.

There is an injective $*$-homomorphism $\psi\colon C^*(G)\to M_n(C^*(H))$ such that the image of $\varphi_n\circ \psi$ is infinite-dimensional. For the convenience of the reader, we sketch the construction of $\psi$, which is standard.

Let $x_1,\dots,x_n\in G$ be such that $x_1=e$ and $G=\sqcup_{i=1}^nx_iH$. Consider the following unitary defined on canonical basis vectors:
\begin{align*}
U\colon\bigoplus_{i=1}^n\ell^2(H)&\to\ell^2(G)\\
\delta_{i,h}&\mapsto\delta_{x_ih}.
\end{align*}

 Let $S\colon B(\ell^2(G))\to M_n(B(\ell^2(H))$ be the isomorphism induced by $U$. 

By using the left regular representations $\lambda_G$ and $\lambda_H$, we see $C^*(G)$ as contained in $B(\ell^2(G))$ and analogously for $C^*(H)$.

It is easy to check that $S(\lambda_G(g))\in M_n(C^*(H))$ for every $g\in G$. Hence, $S(C^*(G))\subset M_n(C^*(H))$. Furthermore, for $h\in H$, we have that $S(\lambda_G(h))_{1,1}=\lambda_H(h)$. Let $\psi:=S|_{C^*(G)}$. Then $\varphi_n( \psi(C^*(G)))$ is infinite-dimensional. 

 Hence, by Lemma \ref{blim}, $M_n(C^*(\mathbb{Z}))\simeq M_n(C(\mathbb{T}))$ contains a sequence of non-zero, orthogonal projections. Since $\mathbb{T}$ is connected, we get a contradiction. Hence, $G$ is finite. 
\end{proof}

\begin{remark}
Recall that a $C^*$-algebra $A$ is said to have property (SP) if every non-zero hereditary $C^*$-subalgebra of $A$ contains a non-zero projection. Furthermore, $A$ is said to have residual property (SP) if every quotient of $A$ has property (SP) (see \cite[Section 7]{MR3352760} for more details about these properties).

In the proof of Theorem \ref{tempo}, the only aspects of real rank zero that were used are that it implies property (SP) and that having real rank zero is closed under taking quotients. In particular, Theorem \ref{tempo} remains true if one replaces ``real rank zero" by ``residual property (SP)".
\end{remark}

\section{Koopman representation of a topological full group}\label{main}

Given a unitary representation $\pi$ of a group $G$, we denote by $C^*_\pi(G)$ the $C^*$-algebra generated by the image of $\pi$.

We will denote the Cantor set by $X$.

Let $\alpha$ be an action of a group $G$ on $X$ by homeomorphisms. The topological full group associated to $\alpha$, denoted by $[[\alpha]]$, is the group of all homeomorphisms $\gamma$ on $X$ for which there exists a finite partition of $X$ into clopen sets $\{A_i\}_{i=1}^n$ and $g_1,\dots,g_n\in G$ such that $\gamma|{A_i}=\alpha_{g_i}|_{A_i}$ for $1\leq i \leq n$. That is, $[[\alpha]]$ consists of the homeomorphisms on $X$ which are locally given by the action $\alpha$. 

Fix $T$ a minimal homeomorphim on $X$. We denote by $[[T]]$ the topological full group associated to the $\mathbb{Z}$-action induced by $T$. 

Let $\mu$ be a $T$-invariant probability measure on $X$. Note that $\mu$ is also invariant under the action of $[[T]]$ on $X$. Let $\pi\colon[[T]]\to B(L^2(X,\mu))$ be given by $\pi(g)(f):=f\circ g^{-1}$, for $g\in [[T]]$ and $f\in L^2(X,\mu)$. This $\pi$ is the so called Koopman representation associated to the action of $[[T]]$ on $(X,\mu)$.

We will use the faithful representation of $C(X)\rtimes\mathbb{Z}$ in $B(L^2(X,\mu))$, with $C(X)$ acting by multiplication operators, and, for n$\in\mathbb{Z}$, $\delta_n:=\pi(T^n)$, so that $C(X)\rtimes\mathbb{Z}:=\overline{\Span}\{f\delta_n:f\in C(X),n\in\mathbb{Z}\}$.  

Given $g\in[[T]]$ and $\{A_i\}_{i=1}^n$ a partition of $X$ into clopen sets such that $g|_{A_i}=T^{n_i}|{A_i}$ for $1\leq i\leq n$, notice that 

\begin{equation}
\pi(g)=\sum 1_{T^{n_i}(A_i)}\delta_{n_i}. \label{deta}
\end{equation}
In particular, $C^*_\pi([[T]])\subset C(X)\rtimes\mathbb{Z}$.

\begin{definition}
Given $n\in\mathbb{N}$, we say that a subset $A\subset X$ is $n$-disjoint if 
$$A, T(A), \dots, T^{n-1}(A)$$
 are pairwise disjoint.
\end{definition}

Suppose $A\subset X$ is a clopen and $n$-disjoint set. Consider the symmetric group $S_n$ acting on $\{0,\dots,n-1\}$. For $\sigma\in S_n$, let $\sigma_A\in[[T]]$ be given by 
\begin{align}\label{sigma}
\sigma_A(x)=\begin{cases}T^{\sigma(i)-i}(x), &\text{if $0\leq i < n$ and $x\in T^i(A)$ }\\
x, & \text{if $x\notin \sqcup_{i=0}^{n-1}T^i(A)$},\end{cases}\quad x\in X.
\end{align}

Note that, for $0\leq i <n$, $\sigma_A(T^i(A))=T^{\sigma(i)}(A)$.

\begin{lemma}\label{sym}
Let $n\geq 4$ and $A\subset X$ be a clopen and $n$-disjoint set. For every $\sigma\in S_n$, it holds that $\pi(\sigma_A)\in C^*_\pi([[T]]')$.
\end{lemma}
\begin{proof}

Notice first that $\{1_{T^i(A)}\delta_{i-j}\}_{0\leq i,j< n}$ forms a system of matrix units in $C(X)\rtimes\mathbb{Z}$ of type $M_n(\mathbb{C})$ (we see $M_n(\mathbb{C})$ as matrices indexed by the set $\{0,\dots,n-1\}$).

  Let $B:=(\sqcup_{i=0}^{n-1}T^i(A))^\mathrm{c}$ and $\varphi\colon \mathbb{C}\oplus M_n(\mathbb{C})\to C(X)\rtimes\mathbb{Z}$ be the $*$-homomorphism given by $\varphi(\alpha,e_{ij}):=\alpha 1_B+1_{T^i(A)}\delta_{i-j}$, for $\alpha\in\mathbb{C}$ and $0\leq i,j\leq n-1$.  
  
  Let $\rho\colon S_n\to \mathbb{C}\oplus M_n(\mathbb{C})$ be the direct sum of the trivial representation and the permutation representation. 

Given $\sigma\in S_n$, by \eqref{deta} and \eqref{sigma}, it holds that $\pi(\sigma_A)=1_B+\sum1_{T^{\sigma(i)}(A)}\delta_{\sigma(i)-i}=\varphi(\rho(\sigma))$.

  Furthermore, since $n\geq 4$, the permutation representations of $S_n'$ and $S_n$ decompose into the direct sum of a trivial representation and an irreducible representation of degree $n-1$. Therefore, we have that $C^*_\rho((S_n)')=C^*_\rho(S_n)$.

Hence, $\pi(\sigma_A)\in  C^*_\pi([[T]]')$ for any $\sigma\in S_n$.

\end{proof}

Given $A\subset X$ clopen, consider the continuous function
\begin{align*}
t_A\colon A&\to\mathbb{N}\\
x&\mapsto\min\{k\geq 1:T^k(x)\in A\}.
\end{align*}
This is the so called function of first return to $A$. 

Notice that, for $j\in\mathbb{Z}$, it holds that 
\begin{equation}\label{fac}
t_{T^j(A)}\circ T^j|_A=t_A.
\end{equation}

  Let $T_A\in[[T]]$ be defined by
  \begin{align}\label{time}
T_A(x)=\begin{cases}T^{t_A(x)}(x), &\text{if $x\in A$ }\\
x, & \text{otherwise}.\end{cases},\quad x\in X.
\end{align}

If $B\subset X$ is a clopen set disjoint from $A$, then $T_A$ and $T_B$ commute.

In order to prove Lemma \ref{tempi}, we will have to analyze the spectrum of $C^*$-algebras generated by certain commuting unitaries, and the next lemma will be useful for this. 

We consider the circle $\mathbb{T}$ as a pointed space with basepoint $1$.

\begin{lemma}\label{univ}
The universal $C^*$-algebra generated by commuting unitaries $z_1,\dots,z_n$ subject to the relations $\{(z_i-1)(z_j-1)=0:1\leq i\neq j\leq n\}$ is $C(\bigvee_{k=1}^n \mathbb{T})$, with $z_k$ being given by
\begin{align*}
  z_k\colon  \bigvee_{i=1}^n \mathbb{T}&\to \mathbb{C} \\
  (x,i)&\mapsto\begin{cases}x, &\text{if } i=k \\
1, & \text{if } i\neq k. \end{cases}
\end{align*}

\end{lemma}
\begin{proof}
Consider the embedding $F\colon \bigvee_{i=1}^n \mathbb{T}\to \mathbb{T}^n$ which takes $x$ in the $i$-th copy of $\mathbb{T}$ and sends it into $(F(x)_i)_{1\leq i \leq n}\in\mathbb{T}^n$ such that $F(x)_i:=x$ and $F(x)_j:=1$ if $j\neq i$. Also let $F'\colon C(\mathbb{T}^n)\to C(\bigvee_{i=1}^n \mathbb{T})$ be given by $F'(f):=f\circ F$, for $f\in C(\mathbb{T}^n)$. 

For $1\leq i\leq n$, let $w_i\in C(\mathbb{T}^n)$ be given by $w_i(y):=y_i$, for $y\in \mathbb{T}^n$. Then $F'(w_i)=z_i$.

Assume $n>1$. Let $A:=C^*(\{(w_i-1)^k(w_j-1)^l: i\neq j\text{ and }k,l\in\mathbb{N}\})$. We claim that $\ker F'=A$. Clearly, $A\subset \ker F'$.

Let $Y:=\mathbb{T}^n\setminus\Ima(F)$. Notice that $\ker F'=\{f\in C(\mathbb{T}^n):f|_{\Ima(F)}=0\}\simeq C_0(Y)$. By the Stone-Weierstrass Theorem, in order to show that $A= C_ 0(Y)$, it is sufficient to show that, for every $y\in Y$, there is $f\in A$ such that $f(y)\neq 0 $, and that $A$ separates the points of $Y$. The proof of the former condition is trivial, so we only show that $A$ separates the points of $Y$.

Take $(x_1,\dots,x_n),(y_1,\dots y_n)\in Y$ distinct points. There is $i$ such that $x_i\neq y_i$. Without loss of generality, assume $x_i\neq 1$. Take $j\neq i$ such that $x_j\neq 1$. Then, by choosing $k\in\mathbb{N}$ appropriately, we get $(x_i-1)^k(x_j-1)\neq(y_i-1)^k(y_j-1)$.

Since $C(\mathbb{T}^n)$ is the universal $C^*$-algebra generated by $n$ commuting unitaries and $ C(\bigvee_{i=1}^n \mathbb{T})$ is generated by $\{z_1,\dots,z_n\}$, the result follows.

\end{proof}

\begin{lemma}\label{tempi}
Let $A\subset X$ be a clopen and $3$-disjoint set. Then $\pi(T_A)\in C^*_\pi([[T]]')$.
\end{lemma}
\begin{proof}

  Given $\sigma\in S_3$, $x\in A$ and $0\leq i,j <3$, we have that $\sigma_A T_{T^i(A)}\sigma_A^{-1}(T^j(x))=T^j(x)$ if $j\neq\sigma(i)$ and 
  
  \begin{align*}\sigma_A T_{T^i(A)}\sigma_A^{-1}(T^{\sigma(i)}(x))&=\sigma_AT_{T^i(A)}(T^i(x))\\
  &=T^{\sigma(i)-i}T^{t_{T^i(A)}(T^i(x))}(T^i(x))\\
  &\stackrel{(*)}=T^{\sigma(i)-i}T^{t_{T^{\sigma(i)}(A)}(T^{\sigma(i)}(x))}(T^i(x))\\
  &=T^{t_{T^{\sigma(i)}(A)}(T^{\sigma(i)}(x))}(T^{\sigma(i)}(x))\\
  &=T_{T^{\sigma(i)}(A)}(T^{\sigma(i)}(x)),
  \end{align*}
 where the equality in (*) is due to \eqref{fac}. Hence, $\sigma_A T_{T^i(A)}\sigma_A^{-1}=T_{T^{\sigma(i)}(A)}$.
  
    In particular, for $0\leq i,j< 3$, we have that $T_{T^i(A)}(T_{T^j(A)})^{-1}\in [[T]]'$.

If $0\leq i\neq j<3$, then $T^i(A)$ and $T^j(A)$ are disjoint, hence $(\pi(T_{T^i(A)})-1)(\pi(T_{T^j(A)})-1)=0$.

 Then, by Lemma \ref{univ}, there is a $*$-homomorphism 

\begin{align*}
\varphi\colon C\left(\bigvee_{i=1}^3 \mathbb{T}\right)&\to C^*(\{\pi(T_{T^i(A)}):0\leq i < 3\}\\
z_i&\mapsto \pi(T_ {T^{i-1}(A)}).
\end{align*}

 Furthermore, by the Stone-Weierstrass theorem, $C(\bigvee_{i=1}^3 \mathbb{T})$ is generated by $$\{z_iz_j^*:1\leq i,j\leq 3\}.$$ Hence, $\pi(T_A)\in C^*_\pi([[T]]')$.

\end{proof}

\begin{theorem}\label{com}
Let $(X,T,\mu)$ be a Cantor minimal system and $\pi$ the Koopman representation associated to the action of $[[T]]$ on $(X,\mu)$. Then $C^*_\pi([[T]])=C^*_\pi([[T]]')$.
\end{theorem}
\begin{proof}
By \cite[Theorem 4.7]{MR3241829}, given $m\in\mathbb{N}$, [[T]] is generated by 
$$\bigcup_{n\geq m}\{T_A,\sigma_A:\sigma\in S_n,\text{$A\subset X$ is clopen and $n$-disjoint}\}.$$ 

By Lemmas \ref{sym} and \ref{tempi}, the result follows.

\end{proof}

Notice that $1_X\in L^2(X,\mu)$ is invariant under $\pi([[T]])$. Therefore, $\pi$ contains the trivial representation. 

\begin{lemma}\label{new}
Let $\rho\colon G\to B(H)$ be a unitary representation which weakly contains the trvial representation, and $\tau$ the associated character on $C^*_\rho(G)$. Then $\ker\tau=\overline{\Span}\{1-\rho(g):g\in G\}$.
\end{lemma}

\begin{proof}
Given $d\in\ker\tau$ and $\epsilon>0$, take $d'\in\Span\rho(G)$ such that $\|d-d'\|<\frac{\epsilon}{2}$. Then $\|d-(d'-\tau(d'))\|=\|(d-d')+\tau(d'-d)\|<\epsilon$. Furthermore, $d'-\tau(d')\in\ker\tau\cap\Span\rho(G)$.

Since $\ker\tau\cap\Span\rho(G)=\Span\{1-\rho(g):g\in G\}$, the result follows.

\end{proof}
\begin{theorem}\label{espera}
Let $\tau$ be the character on $C^*_\pi([[T]])$ coming from containment of the trivial representation. Then $\ker\tau$ is a hereditary $C^*$-subalgebra of $C(X)\rtimes\mathbb{Z}$.
\end{theorem}
\begin{proof}
We are going to show that, for $a\in C(X)\rtimes\mathbb{Z}$ and $b,c\in\ker\tau$, it holds that $bac\in\ker\tau$. 

Given $A\subset X$ clopen and $2$-disjoint, notice that $(\delta_0-\delta_1)1_A(\delta_0-\delta_{-1})=\delta_0-(1_{(A\cup T(A))^\text{c}}\delta_0+1_{T(A)}\delta_1+1_A\delta_{-1})\in C^*_\pi([[T]])$.

By using telescoping sums, it follows that, for $n,m\in\mathbb{Z}$ and $A\subset X$ $2$-disjoint and clopen, $(\delta_0-\delta_n)1_A(\delta_0-\delta_m)\in C^*_\pi([[T]])$.

Given $g,h\in[[T]]$, take a basis $\mathcal{B}$ of 2-disjoint, clopen sets for the topology of $X$. Moreover, assume that, for each $A\in\mathcal{B}$, there is $n(A),m(A)\in\mathbb{Z}$ such that $g|_A=T^{n(A)}|_A$ and $h|_{h^{-1}(A)}=T^{m(A)}|_{h^{-1}(A)}$.  

Then 
\begin{align*}
(\delta_0-\pi(g))1_A(\delta_0-\pi(h))&=1_A-\delta_{n(A)}1_A-1_A\delta_{m(A)}+\delta_{n(A)}1_A\delta_{m(A)}\\
&=(\delta_0-\delta_{n(A)})1_A(\delta_0-\delta_{m(A)})\in C^*_\pi([[T]]).
\end{align*}

Since $C(X)=\overline{\Span}\{1_A:A\in\mathcal{B}\}$, we conclude that, for $g,h\in[[T]]$ and $f\in C(X)$, $(\delta_0-\pi(g))f(\delta_0-\pi(h))\in C^*_\pi([[T]]).$

By Lemma \ref{new} and the fact that $C(X)\rtimes\mathbb{Z}=\overline{\Span}\{f\delta_n:f\in C(X),n\in\mathbb{Z}\}$, we conclude that, for $b,c\in\ker\tau$ and $a\in C(X)\rtimes\mathbb{Z}$, $bac\in C^*_\pi([[T]])$. 

Since $\tau$ is a character, the result follows.
\end{proof}
\begin{corollary}
 Let $\tau$ be the character on $C^*_\pi([[T]])$ coming from containment of the trivial representation. Then $\ker\tau$ is stably isomorphic to $C(X)\rtimes\mathbb{Z}$. In particular, $C^*_\pi([[T]]')$ has real rank zero and $C^*([[T]]')$ is not AF.
\end{corollary}
\begin{proof}
By Theorem \ref{espera} and the fact that $C(X)\rtimes\mathbb{Z}$ is simple, it follows that $\ker\tau$ is a full, hereditary $C^*$-subalgebra of $C(X)\rtimes\mathbb{Z}$. Therefore, \cite[Theorem 2.8]{MR0454645} implies that $\ker\tau$ is stably isomorphic to $C(X)\rtimes\mathbb{Z}$.

Furthermore, by Theorem \ref{com}, $C^*_\pi([[T]])=C^*_\pi([[T]]')$. Since $C(X)\rtimes\mathbb{Z}$ has real rank zero (see, for instance, \cite{MR2134336} for a proof of this fact), and $K_1(C(X)\rtimes\mathbb{Z})\simeq\mathbb{Z}$, and $K_1(A)=0$ for any AF-algebra $A$, the conclusion follows.
\end{proof}

\section{Odometers}\label{odom}

We start this section by giving a description of $C^*_\pi([[T]])$ when $T$ is an odometer map.

Given $m\in\mathbb{N}$, let $\mathbb{Z}_m:=\mathbb{Z}/m\mathbb{Z}$.
\begin{example}\label{odo}

Let $(n_k)$ be a strictly increasing sequence of natural numbers such that, for every $k$, $n_k|n_{k+1}$. Let $\rho_k\colon \mathbb{Z}_{n_{k+1}}\to\mathbb{Z}_{n_k}$ be the surjective homomorphism such that $\rho_k(1)=1$, and define 
$$X:=\{(x_k)\in \prod_{k\in\mathbb{N}} \mathbb{Z}_{n_k}: \rho_{k}(x_{k+1})=x_k, \forall k\in\mathbb{N}\}.$$

Consider
\begin{align*}
T\colon X&\to X\\
(x_k)&\mapsto (x_k+1).
\end{align*}
 Then $(X,T)$ is a Cantor minimal system, the so called odometer of type $(n_k)$.

For $k\in\mathbb{N}$ and $l\in\mathbb{Z}_{n_k}$, let $U(k,l):=\{(x_m)\in X:x_k=l\}$. 

Using the notation from \eqref{sigma} and \eqref{time}, let, for $k\in\mathbb{N}$, $\Gamma_k:=\langle\{ T_{U(k,l)},\sigma_{U(k,0)}\in[[T]]:l\in\mathbb{Z}_{n_k},\sigma\in S_{n_k}\}\rangle$. As proven by Matui in \cite[Proposition 2.1]{MR3103094}, $\Gamma_k\subset \Gamma_{k+1}$, $\Gamma_k\simeq\mathbb{Z}^{n_k}\rtimes S_{n_k}$, and $\bigcup_k\Gamma_k=[[T]]$.

For $k\in\mathbb{N}$, let $A_k:=\overline{\Span}\{1_{U(k,l)}\delta_m:l\in\mathbb{Z}_{n_k},m\in\mathbb{Z}\}.$ Then $A_k\subset A_{k+1}$, and $C(X)\rtimes \mathbb{Z}=\overline{\bigcup_k A_k}$.

Fix $k\in\mathbb{N}$ and consider the isomorphism $\varphi_k\colon A_k\to C(\mathbb{T},M_{\mathbb{Z}_{n_k}}(\mathbb{C}))$, such that $\varphi_k(1_{U(k,l)})=e_{l,l}$, for $l\in\mathbb{Z}_{n_k}$, and, for $z\in\mathbb{T}$,  
\begin{align*}
 (\varphi_k(\delta_1)(z))_{i,j}:=\begin{cases}1, &\text{if } 0< i\leq n_k-1\text{ and }j=i-1 \\
z, & \text{if } i=0\text{ and }j=n_k-1\\
0, & \text{otherwise}. \end{cases}
\end{align*}

Let $\pi\colon[[T]]\to U(C(X)\rtimes\mathbb{Z})$ be the homomorphism coming from the Koopman representation and $B_k:=\{b\in M_{\mathbb{Z}_{n_k}}(\mathbb{C}):\forall i,j\in\mathbb{Z}_{n_k},\sum_r b_{i,r}=\sum_s b_{s,j}\}$. 

Then, for $\sigma\in S_{n_k}$, we have that $\varphi_k(\pi(\sigma_{U(k,0)}))=\sum e_{\sigma(i),i}$ and $$C^*(\{\varphi_k(\pi(\sigma_{U(k,0)})):\sigma\in S_{n_k}\})\simeq B_k.$$

  Furthermore, $\varphi_k(C^*(\pi(\{  T_{U(k,l)}:l\in\mathbb{Z}_{n_k}\})))\simeq C(\bigvee_{l\in\mathbb{Z}_{n_k}} \mathbb{T})$ and $\varphi_k(C^*_\pi(\Gamma_k))=\{f\in C(\mathbb{T},M_{\mathbb{Z}_{n_k}}(\mathbb{C})):f(1)\in B_k\}$.

\end{example}

In \cite{MR1759493}, Dykema and Rørdam gave examples of non-locally finite groups $G$ such that $C^*_{\mathrm{red}}(G)$ has real rank zero. As far as we are aware, there is no known example of non-locally finite group $G$ such that $C^*(G)$ has real rank zero. 

\begin{question}
Let $(X,T)$ be an odometer as in Example \ref{odo}. Does $C^*([[T]]')$ have real rank zero?
\end{question}

\begin{example}\label{sonho}
Let $(X,T)$ be an odometer of type $(n_k)$ as in Example \ref{odo}. Consider 

\begin{align*}
J\colon X&\to X\\
(x_k)&\mapsto (-x_k).
\end{align*}

Then $J$ is an involutive homeomorphism on $X$ such that $JTJ=T^{-1}$. Hence, $T$ and $J$ induce an action $\alpha$ of the infinite dihedral group $\mathbb{Z}\rtimes\mathbb{Z}_2$ on $X$. 
 We will use Matui's technique (\cite[Proposition 2.1]{MR3103094}) in order to compute $[[\alpha]]$.
 
For every $\gamma\in(\mathbb{Z}\rtimes\mathbb{Z}_2)\setminus\{e\}$, it holds that $\{x\in X:\alpha_\gamma(x)=x\}$ has empty interior (it consists of at most two elements). Hence, given $g\in[[\alpha]]$, there exists a unique continuous function $c_g\colon X\to\mathbb{Z}\rtimes\mathbb{Z}_2$ such that, for $x\in X$, $g(x)=\alpha_{c(g)}(x)$.  
 
 For $k\in\mathbb{N}$ and $l\in\mathbb{Z}_{n_k}$, let $U(k,l)$ be as in Example \ref{odo} and
 $$\Gamma_k:=\{g\in [[\alpha]]:\text{$c_g$ is constant on $U(k,l)$ for $l\in\mathbb{Z}_{n_k}$}\}.$$
 
 Define $J_{k,l}\in[[\alpha]]$ by
 \begin{align*}
J_{k,l}(x)=\begin{cases}T^{2l}J(x), &\text{if $x\in U(k,l)$ }\\
x, & \text{otherwise},\end{cases}\quad x\in X.
\end{align*}

Then $\Gamma_k=\langle \{T_{U(k,l)},J_{k,l},\sigma_{U(k,0)}:l\in\mathbb{Z}_{n_k},\sigma\in S_{n_k}\}\rangle$ and \begin{align}\label{nao}
\Gamma_k\simeq(\mathbb{Z}\rtimes\mathbb{Z}_2)^{n_k}\rtimes S_{n_k},\text{ } \Gamma_k\subset\Gamma_{k+1},\text{ and }\bigcup_k\Gamma_k=[[\alpha]].
\end{align}

Notice that the constant sequence $(0)\in X$ is a fixed point for $J$. Hence, \cite[Theorem 3.5]{MR1245825} implies that $C(X)\rtimes(\mathbb{Z}\rtimes\mathbb{Z}_2)$ is AF (see also \cite{MR962104}). Moreover, it follows from \eqref{nao} that the abelianization of $[[\alpha]]$ is locally finite.

 Therefore, the two obstructions that were used for ruling out the possibility of $C^*([[T]])$ and $C^*([[T]]')$ being AF do not hold for $C^*([[\alpha]])$. 
\end{example}
\begin{question}
Let $\alpha$ be as in Example \ref{sonho}. Is $C^*([[\alpha]])$ AF?.
\end{question}

\section*{Acknowledgements}
Part of this work was carried out while the author was attending the research program \textit{Classification of operator algebras: complexity, rigidity, and dynamics} at the Mittag-Leffler Institute. The author thanks the organizers of the program and the staff of the institute for the excellent work conditions provided.

The author also thanks J. Carrión, T. Giordano, K. Li and J. Rout for helpful conversations related to topics of this work.

\bibliographystyle{acm}
\bibliography{bibliografia}
\end{document}